\documentclass[reqno,a4paper,draft]{amsart}

\usepackage{enumitem}
\setenumerate{label=\textnormal{(\arabic*)}}

\usepackage{amsmath,amssymb,dsfont,verbatim,bm,array,mathtools}
\usepackage[latin1]{inputenc}
\usepackage{booktabs}
\usepackage[raggedright]{titlesec}
\usepackage{mathtools}

\titleformat{\chapter}[display]
{\normalfont\huge\bfseries}{\chaptertitlename\\thechapter}{20pt}{\Huge}
\titleformat{\section}
{\normalfont\Large\bfseries\center}{\thesection}{1em}{}
\titleformat{\subsection}
{\normalfont\large\bfseries}{\thesubsection}{1em}{}
\titleformat{\subsubsection}[runin]
{\normalfont\normalsize\bfseries}{\thesubsubsection}{1em}{}
\titleformat{\paragraph}[runin]
{\normalfont\normalsize\bfseries}{\theparagraph}{1em}{}
\titleformat{\subparagraph}[runin]
{\normalfont\normalsize\bfseries}{\thesubparagraph}{1em}{}
\titlespacing*{\chapter} {0pt}{50pt}{40pt}
\titlespacing*{\section} {0pt}{3.5ex plus 1ex minus .2ex}{2.3ex plus .2ex}
\titlespacing*{\subsection} {0pt}{3.25ex plus 1ex minus .2ex}{1.5ex plus .2ex}
\titlespacing*{\subsubsection}{0pt}{3.25ex plus 1ex minus .2ex}{1.5ex plus .2ex}
\titlespacing*{\paragraph} {0pt}{3.25ex plus 1ex minus .2ex}{1em}
\titlespacing*{\subparagraph} {\parindent}{3.25ex plus 1ex minus .2ex}{1em}
\input xypic
\xyoption{all}

 %Secondary 16S32}
\newtheorem{theorem}{Theorem}[section]
\newtheorem{lemma}[theorem]{Lemma}
\newtheorem{proposition}[theorem]{Proposition}

\newtheorem{conjecture}[theorem]{Conjecture}

\theoremstyle{definition}
\newtheorem{definition}[theorem]{Definition}

\theoremstyle{remark}
\newtheorem{remark}[theorem]{Remark}
\newtheorem{remarks}[theorem]{Remarks}

\DeclareMathOperator{\Jac}{Jac}

\begin{document}
\title{About Dixmier's conjecture}

\author{Vered Moskowicz}
\address{Department of Mathematics, Bar-Ilan University, Ramat-Gan 52900, Israel.}
\email{vered.moskowicz@gmail.com}
\thanks{The author was partially supported by an Israel-US BSF grant \#2010/149}

\begin{abstract}
The well-known Dixmier conjecture asks if every algebra endomorphism of the first Weyl algebra over a characteristic zero field is an automorphism.

We bring a hopefully easier to solve conjecture, called the $\gamma,\delta$ conjecture, and show that it is equivalent to the Dixmier conjecture.

Up to checking that in the group generated by automorphisms and anti-automorphisms of $A_1$ all the involutions belong to one conjugacy class, we show that: \begin{itemize}
\item Every involutive endomorphism from $(A_1,\gamma)$ to $(A_1,\delta)$ is an automorphism 
($\gamma$ and $\delta$ are two involutions on $A_1$).
\item Given an endomorphism $f$ of $A_1$ (not necessarily an involutive endomorphism), if one of $f(X)$,$f(Y)$ is symmetric or skew-symmetric (with respect to any involution on $A_1$), then $f$ is an automorphism.
\end{itemize}

\end{abstract}

\maketitle

\section{Introduction}
Throughout this paper, $K$ is a characteristic zero field. 
Let $A_1= A_1(K)=K \langle X,Y | YX-XY= 1 \rangle$ be the first Weyl algebra over $K$. 
Let $\alpha$ be the exchange involution on $A_1$ given by $\alpha(X)= Y$ and $\alpha(Y)= X$.
In \cite[Definition 1.1]{moskowicz valqui} the notion of an $\alpha$-endomorphism of $A_1$ was defined, namely, $f$ is an $\alpha$-endomorphism of $A_1$ if $f$ is an algebra endomorphism of $A_1$ such that $f \circ \alpha= \alpha \circ f$.
The $\alpha$-Dixmier conjecture ($\alpha-D_1$) was posed and proved: Every $\alpha$-endomorphism of $A_1$ is an $\alpha$-automorphism of $A_1$, \cite[Theorem 2.9]{moskowicz valqui}.
Actually, it was shown that every $\alpha$-endomorphism of $A_1$ is of the following form:
$$
f(x):= ax+ by+ \sum_{j= 0}^n S_j\quad\text{and}\quad f(y):= ay+ bx+ \sum_{j= 0}^n S_j.
$$
where $n\in \mathds{N}_0$, $a,b,c_0,\dots,c_n\in K$, with $a^2 - b^2 = 1$, and the $S_j$'s are defined as follows
$S_j:=c_{j}(X-Y)^{2j}$.

Let $K[x,y]$ be the polynomial ring over $K$ and let $\alpha$ be the exchange involution on $K[x,y]$ given by $\alpha(x)= y$ and $\alpha(y)= x$.
The $\alpha$-Jacobian conjecture ($\alpha-JC_{2}$) was posed and proved: If $f:K[x,y]\to K[x,y]$ is an $\alpha$-morphism ($f \circ \alpha= \alpha \circ f$) that satisfies
$\Jac(f(x),f(y))=1$,
then $f$ is invertible, \cite[Proposition 4.1]{moskowicz valqui}.

It is unknown whether $\alpha-JC_{2}$ implies $\alpha-D_1$, as was mentioned there: ``In the previous section we proved the starred Dixmier conjecture (in dimension 1)
and in this section we will prove
the starred Jacobian conjecture in dimension two,
however we do not know if
one can deduce the former directly from the latter as in the star-free case".

\begin{remark}\label{remark starred alpha}
The exchange involution $\alpha$ may be denoted by $X^*= Y$ and $Y^*= X$ instead of by $\alpha(X)= Y$ and $\alpha(Y)= X$; the first notaion justifies the name ``starred Dixmier conjecture", while the second notation justifies the name ``$\alpha$-Dixmier conjecture".
\end{remark}

%%%%%%%%%%%%%%%%%%%%%%%%%%%%%%%%
\section{Equivalence of the Dixmier conjecture and the $\gamma,\delta$ conjecture for $A_1$ (first version)}

Of course, by an $\alpha$-anti-endomorphism of $A_1$ we mean an anti-endomorphism of $A_1$ that commutes with $\alpha$.
The following proposition will be used in our results.
%in the proof of Lemma \ref{gamma,delta is onto}
%in the discussion after Proposition \ref{gamma,delta iff hfg alpha} and in the proof of Theorem \ref{dixmier iff gamma,delta second version}.

\begin{proposition}\label{alpha anti-endomorphisms are onto}
Let $f$ be an $\alpha$-anti-endomorphism of $A_1$. Then $f$ is an $\alpha$-anti-automorphism of $A_1$.
In fact, $f$ is of the following form:
$$
f(x):= ax+ by+ \sum_{j= 0}^n S_j\quad\text{and}\quad f(y):= ay+ bx+ \sum_{j= 0}^n S_j.
$$
where $n\in \mathds{N}_0$, $a,b,c_0,\dots,c_n\in K$, with $a^2 - b^2 = -1$, and the $S_j$'s are defined as follows
$S_j:=c_{j}(X-Y)^{2j}$.
\end{proposition}

Notice that in the form of an $\alpha$-endomorphism we have $a^2 - b^2 = 1$, while in the form of an $\alpha$-anti-endomorphism we have $a^2 - b^2 = -1$.

\begin{proof}
This result can be obtained in a similar way to the way the result concerning $\alpha$-endomorphisms was obtained in \cite{moskowicz valqui}.
\end{proof}

In \cite{moskowicz valqui} only two involutions were considered: $\alpha$ and $\beta$, where $\beta$ is given by $\beta(X)=X$ and
$\beta(Y)= -Y$. 
Clearly, $\beta$ is conjugate to $\alpha$ by $\varphi$: $\beta:=\varphi^{-1}\circ \alpha \circ \varphi$,
where $\varphi(X):=\displaystyle\frac{X+Y}2$ and
$\varphi(Y):= Y-X$.
Here we shall consider more involutions:

\begin{definition}[A $\gamma,\delta$-endomorphism]\label{definition gamma,delta endo}
Let $f$ be an endomorphism of $A_1$ and let $\gamma$ and $\delta$ be any two involutions on $A_1$.
We say that $f$ is \begin{itemize}
\item a $\gamma$-endomorphism of $A_1$, if $f \circ \gamma= \gamma \circ f$.
\item a $\gamma,\delta$-endomorphism of $A_1$, if $f \circ \gamma= \delta \circ f$.
\end{itemize}
(If $\delta= \gamma$, then a $\gamma,\gamma$-endomorphism is just a $\gamma$-endomorphism).
\end{definition}

\begin{remarks}\label{remarks after definition}
\begin{itemize}
\item [(1)] In the above definition we are not assuming that the involutions $\gamma$ and $\delta$ are conjugate to $\alpha$ by an automorphism of $A_1$ or by an anti-automorphism of $A_1$.
However, in the first version of the $\gamma,\delta$ conjecture \ref{first version of the gamma,delta conjecture}, 
we are assuming that each of $\gamma$ and $\delta$ is conjugate to $\alpha$.

\item [(2)] Every automorphism $f$ of $A_1$ is a $\gamma,\delta$-endomorphism (with each of $\gamma$ and $\delta$ conjugate to $\alpha$); just take $\gamma= f^{-1} \alpha f$ , $\delta= \alpha$ and get $f \gamma= f(f^{-1} \alpha f)= \alpha f= \delta f$.

\item [(3)] If $f$ is a $\gamma,\delta$-endomorphism, then $\gamma$ and $\delta$ are not necessarily unique such that 
$f \circ \gamma= \delta \circ f$. 
For example: If $f$ is an automorphism of $A_1$, then the following are legitimate candidates for $\gamma$ and $\delta$:
\begin{itemize}
\item [] $\gamma= f^{-1} \alpha f$, $\delta= \alpha$.
\item [] $\gamma= f^{-2} \alpha f^2$, $\delta= f^{-1} \alpha f$.
\item [] $\gamma= f \alpha f^{-1}$, $\delta= f^2 \alpha f^{-2}$.
\end{itemize}
It is easy to see that there exist infinitely many other options for $\gamma$ and $\delta$.

\end{itemize}
\end{remarks}

\begin{conjecture}[The $\gamma,\delta$ conjecture (first version)]\label{first version of the gamma,delta conjecture}
For every endomorphism $f$ of $A_1$, there exist involutions $\gamma$ and $\delta$, each of $\gamma$ and $\delta$ is conjugate to $\alpha$, such that $f$ is a $\gamma,\delta$-endomorphism of $A_1$.
\end{conjecture}

The following lemma shows why it is good for an endomorphism of $A_1$ to be a $\gamma,\delta$-endomorphism, where each of $\gamma$ and $\delta$ is conjugate to $\alpha$.

\begin{lemma}\label{gamma,delta is onto}
Let $f$ be an endomorphism of $A_1$.
Then: $f$ is a $\gamma,\delta$-endomorphism of $A_1$, where each of $\gamma$ and $\delta$ is conjugate to $\alpha$ $\Longleftrightarrow$ $f$ is an automorphism of $A_1$.
\end{lemma}

\begin{proof}
$\Longleftarrow$: This is just the second remark of Remarks \ref{remarks after definition}.

$\Longrightarrow$: $f$ is a $\gamma,\delta$-endomorphism of $A_1$, so $f \gamma= \delta f$.

There exists an automorphism or an anti-automorphism $g$ of $A_1$, such that $\gamma= g^{-1} \alpha g$, and there exists an automorphism or an anti-automorphism $h$ of $A_1$, such that $\delta= h^{-1} \alpha h$.

%$\gamma$ is conjugate to $\alpha$ means that there exists an automorphism or an anti-automorphism $g$ of $A_1$, such that $\gamma= g^{-1} \alpha g$. $\delta$ is conjugate to $\alpha$ means that there exists an automorphism or an anti-automorphism $h$ of $A_1$, such that $\delta= h^{-1} \alpha h$.

Therefore, $f \gamma= \delta f$ becomes $f g^{-1} \alpha g= h^{-1} \alpha h f$, 
so $(h f g^{-1}) \alpha= \alpha (h f g^{-1})$.

\begin{itemize}
\item If both $g$ and $h$ are automorphisms or anti-automorphisms, then $h f g^{-1}$ is an $\alpha$-endomorphism, hence by \cite[Theorem 2.9]{moskowicz valqui} $h f g^{-1}$ is an $\alpha$-automorphism.

Then obviously $f$ is an automorphism; indeed, $f= (h^{-1}h)f(g^{-1}g)= h^{-1}(hfg^{-1})g$, so $f$ is a composition of three automorphisms: $h^{-1}$, $hfg^{-1}$ and $g$, or $f$ is a composition of two anti-automorphisms $h^{-1}$ and $g$ and an automorphism $hfg^{-1}$.

\item If one of $g$ and $h$ is an automorphism and the other is an anti-automorphism, then $h f g^{-1}$ is an $\alpha$-anti-endomorphism, hence by Proposition \ref{alpha anti-endomorphisms are onto} $h f g^{-1}$ is an $\alpha$-anti-automorphism.

Then obviously $f$ is an automorphism (as a composition of two anti-automorphisms and one automorphism).
\end{itemize}
\end{proof}
%%%%%%%%%%%%%%%%%
Of course, our definition of a $\gamma,\delta$-endomorphism $f$ of $A_1$ is just another way to say that $f$ is an endomorphism from the ring with involution $(A_1,\gamma)$ to the ring with involution $(A_1,\delta)$ (since, by definition, an endomorphism from the ring with involution $(A_1,\gamma)$ to the ring with involution $(A_1,\delta)$ is an endomorphism of $A_1$ that satisfies $f \circ \gamma= \delta \circ f$).

So the non-trivial direction of Lemma \ref{gamma,delta is onto} just says the following:

\begin{theorem}[The involutive Dixmier conjecture is true (first version)]
Let $f$ be an endomorphism from the ring with involution $(A_1,\gamma)$ to the ring with involution $(A_1,\delta)$, where each of $\gamma$ and $\delta$ is conjugate to $\alpha$.
Then $f$ is an automorphism.
\end{theorem}
%%%%%%%%%%%%%%%%%
The computations in the proof of Lemma \ref{gamma,delta is onto} are yielding the following trivial proposition. 
\begin{proposition}\label{gamma,delta iff hfg alpha}
Let $f$ be an endomorphism of $A_1$ and let $\gamma$ and $\delta$ be two involutions on $A_1$, 
such that $\gamma= g^{-1}\alpha g$ and $\delta= h^{-1}\alpha h$, where both $g$ and $h$ are automorphisms or both are anti-automorphisms or one of $g$ and $h$ is an automorphism and the other is an anti-automorphism.

Then the following statements are equivalent:\begin{itemize}
\item [(1)] $f$ is a $\gamma,\delta$-endomorphism of $A_1$.
\item [(2)] $h f g^{-1}$ is an $\alpha$-endomorphism (when both $g$ and $h$ are automorphisms or both are anti-automorphisms) or an $\alpha$-anti-endomorphism (when one of $g$ and $h$ is an automorphism and the other is an anti-automorphism).
\end{itemize}
\end{proposition}

\begin{proof}
Clear.
\end{proof}

Given an endomorphism $f$ of $A_1$, each of the statements of Proposition \ref{gamma,delta iff hfg alpha} is equivalent to $f$ being an automorphism: \begin{itemize}

\item [(1)] Lemma \ref{gamma,delta is onto} shows that being a $\gamma,\delta$-endomorphism of $A_1$ (with each of $\gamma$ and $\delta$ conjugate to $\alpha$) is equivalent to being an automorphism of $A_1$.

\item [(2)] First option: $h f g^{-1}$ is an $\alpha$-endomorphism (both $g$ and $h$ are automorphisms or both are anti-automorphisms), then from \cite[Theorem 2.9]{moskowicz valqui} $h f g^{-1}$ is an $\alpha$-automorphism, hence $f$ is an automorphism.

Second option: $h f g^{-1}$ is an $\alpha$-anti-endomorphism (one of $g$ and $h$ is an automorphism and the other is an anti-automorphism), then from Proposition \ref{alpha anti-endomorphisms are onto} $h f g^{-1}$ is an $\alpha$-anti-automorphism, hence $f$ is an automorphism.

The opposite direction is trivial: If $f$ is an automorphism of $A_1$, then taking $g=f$ and $h=1$ yields
$h f g^{-1}= 1 f f^{-1}= 1$ which is an $\alpha$-endomorphism.

\end{itemize}

In other words, given an endomorphism $f$ of $A_1$, in order to get a positive answer to the Dixmier conjecture, it is sufficient to show one of the following:\begin{itemize}
\item [(1)] $f$ is a $\gamma,\delta$-endomorphism of $A_1$ (with each of $\gamma$ and $\delta$ conjugate to $\alpha$).

\item [(2)] There exist $g$ and $h$ both are automorphisms or both are anti-
automorphisms of $A_1$ such that $h f g^{-1}$ is an $\alpha$-endomorphism of $A_1$; 

or there exist $g$ and $h$, one of $g$ and $h$ is an automorphism of $A_1$ and the other is an anti-automorphism of $A_1$ such that $h f g^{-1}$ an $\alpha$-anti-endomorphism of $A_1$.
\end{itemize}

Summarizing, the connection between the Dixmier conjecture and the $\gamma,\delta$ conjecture (first version) is as follows:
\begin{theorem}[Equivalence of the Dixmier conjecture and the $\gamma,\delta$ conjecture (first version)]\label{dixmier iff gamma,delta first version}
The Dixmier conjecture is true $\Longleftrightarrow$ The $\gamma,\delta$ conjecture (first version) is true.
\end{theorem}

\begin{proof}
$\Longrightarrow$: Let $f$ be an endomorphism of $A_1$. We must find involutions $\gamma$ and $\delta$, each of $\gamma$ and $\delta$ is conjugate to $\alpha$, such that $f \circ \gamma= \delta \circ f$.
The Dixmier conjecture is true, hence $f$ is an automorphism, so use Lemma \ref{gamma,delta is onto}.

$\Longleftarrow$: Let $f$ be an endomorphism of $A_1$. We must show that $f$ is an automorphism of $A_1$.
The $\gamma,\delta$ conjecture (first version) is true, hence there exist involutions $\gamma$ and $\delta$, each of $\gamma$ and $\delta$ is conjugate to $\alpha$, such that $f$ is a $\gamma,\delta$-endomorphism of $A_1$.
Now use Lemma \ref{gamma,delta is onto}. 
\end{proof}

And similarly,
\begin{conjecture}[The $g,h$ conjecture (first version)]\label{first version of the g,h conjecture}
For every endomorphism $f$ of $A_1$, there exist $g$ and $h$, both automorphisms, both anti-automorphisms or one is an automorphism and the other is an anti-automorphism, such that $h f g^{-1}$ is an $\alpha$-endomorphism (when both $g$ and $h$ are automorphisms or both are anti-automorphisms) or an $\alpha$-anti-endomorphism (when one of $g$ and $h$ is an automorphism and the other is an anti-automorphism).
\end{conjecture}

\begin{theorem}[Equivalence of the Dixmier conjecture and the $g,h$ conjecture (first version)]\label{dixmier iff g,h first version}
The Dixmier conjecture is true $\Longleftrightarrow$ The $g,h$ conjecture (first version) is true.
\end{theorem}

\subsection{An extension condition and a restriction condition (first version)}

In view of the above results, our (hopefully possible) mission is to prove that the $\gamma,\delta$ conjecture (first version) is true or that the $g,h$ conjecture (first version) is true; Thus far, given an endomorphism $f$ of $A_1$, we only managed to find two conditions, an extension condition and a restriction condition, each implies that $f$ is a $\gamma,\delta$-endomorphism of $A_1$ (each of $\gamma$ and $\delta$ is conjugate to $\alpha$), hence an automorphism of $A_1$, see Theorem \ref{theorem of extension condition first version} and Theorem \ref{theorem of restriction condition first version}.

We have not yet managed to find a condition that implies directly that the $g,h$ conjecture (first version) is true: We only suggest that instead of considering a general endomorphism $f$ of $A_1$, one will consider an irreducible endomorphism $f$ of $A_1$, see \cite[Definition 3.1]{GGV.old} for the definition of an irreducible endomorphism. 

Namely, we suggest that for every irreducible endomorphism $f$ of $A_1$, one will find $g$ and $h$ such that $h f g^{-1}$ is an $\alpha$-endomorphism or an $\alpha$-anti-endomorphism of $A_1$. Hence, $h f g^{-1}$ is onto, so $f$ is an automorphism. Therefore, we have shown that every irreducible endomorphism of $A_1$ is an automorphism of $A_1$, concluding that the Dixmier conjecture is true, by changing a little \cite[Theorem 3.3]{GGV.old}.

Now, let $f$ be an endomorphism of $A_1$. 
{}From now on we will use the following notations: 
$P:= f(X)$ and $Q:= f(Y)$. So, $1= [f(Y),f(X)]= [Q,P]$.
The image of $f$, $K \langle P,Q | QP-PQ= 1 \rangle$, is a subalgebra of $A_1$ which is isomorphic to $A_1$; we will denote the image of $f$ by $T$.

We can define an involution $\gamma$ on $T$ by $\gamma(P)= Q$, $\gamma(Q)= P$ (and extended in the obvious way to all of $T$).

U. Vishne raised the following question, as a general question about involutions on $A_1$, without intending to connect it to a solution to the Dixmier conjecture: ``Given an endomorphism $f$ of $A_1$, with $P:= f(X)$ and $Q:= f(Y)$, 
can the involution $\gamma$ ($\gamma(P)= Q$, $\gamma(Q)= P$) on $T$ be extended to $A_1$?".

(If the given endomorphism $f$ of $A_1$ is already an automorphism of $A_1$, then $\gamma$ is already an involution on $A_1$).

We suggest to adapt Vishne's question in a way that connects it to a solution to the Dixmier conjecture.
This is done in Theorem \ref{theorem of extension condition first version}. 

\begin{definition}[The extension condition (first version)]
Let $f$ be an endomorphism of $A_1$. 
We say that $f$ satisfies the extension condition (first version) if the involution $\gamma$ ($\gamma(P)= Q$, $\gamma(Q)= P$) on $T$ can be extended to $A_1$,
where the extended involution is conjugate to $\alpha$.
\end{definition}

\begin{theorem}\label{theorem of extension condition first version}
Let $f$ be an endomorphism of $A_1$.
If $f$ satisfies the extension condition (first version), then $f$ is an automorphism of $A_1$.
\end{theorem}

\begin{proof}
$f$ satisfies the extension condition (first version), so there exists an involution on $A_1$, call it $\Gamma$, such that
$\Gamma$ extends $\gamma$, and $\Gamma$ is conjugate to $\alpha$.

Therefore, we get that $f$ is an $\alpha,\Gamma$-endomorphism of $A_1$ ($\Gamma$ is conjugate to $\alpha$), because:
$(f \alpha)(X)= f(Y)= Q= \gamma(P)= \Gamma(P)= \Gamma(f(X))= (\Gamma f)(X)$ and
$(f \alpha)(Y)= f(X)= P= \gamma(Q)= \Gamma(Q)= \Gamma(f(Y))= (\Gamma f)(Y)$.

Hence, from Lemma \ref{gamma,delta is onto} we get that $f$ is an automorphism of $A_1$.
\end{proof}
%%%%%%%%%%%%%%%%%%%%%%%%%%%%%%%%%%%%%%%%%%%%%%%%%%%%%%%%
We also suggest to consider the following ``restriction condition". 

\begin{definition}[The restriction condition (first version)]
Let $f$ be an endomorphism of $A_1$. 
We say that $f$ satisfies the restriction condition (first version) if the exchange involution $\alpha$ on $A_1$, when restricted to $T$, is an involution on $T$, and $\gamma$ is conjugate by an automorphism or by an anti-automorphism of $T$ to the restriction of 
$\alpha$ to $T$.
\end{definition}

Notice that, in general, the restriction of $\alpha$ to an affine subalgebra of $A_1$ need not be an involution on the affine subalgebra, since $\alpha$ may take the generators of the affine subalgebra outside it.

\begin{theorem}\label{theorem of restriction condition first version}
Let $f$ be an endomorphism of $A_1$.
If $f$ satisfies the restriction condition (first version), then $f$ is an automorphism of $A_1$.
\end{theorem}

\begin{proof}
$f$ satisfies the restriction condition (first version), so $\alpha$ restricted to $T$ is an involution on $T$, 
and $\gamma$ is conjugate by an automorphism or by an anti-automorphism of $T$ to (the restriction of) $\alpha$:

$\gamma= g^{-1} \alpha g$, where $g$ is an automorphism or an anti-automorphism of $T$ ($\alpha$ is the restriction of $\alpha$ to $T$).

Therefore, we get that $f$ is an $\alpha,\gamma$-endomorphism of $A_1$ 
($\gamma$ is conjugate to $\alpha$), because:
$(f \alpha)(X)= f(Y)= Q= \gamma(P)= \gamma(f(X))= (\gamma f)(X)$ and
$(f \alpha)(Y)= f(X)= P= \gamma(Q)= \gamma(f(Y))= (\gamma f)(Y)$.

Hence, from Lemma \ref{gamma,delta is onto} we get that $f$ is an automorphism of $A_1$.
\end{proof}

\begin{remark}\label{remark used in generalized}
In the proof of Theorem \ref{theorem of restriction condition first version}, $\gamma$ is conjugate by $g$ to the restriction of $\alpha$ to $T$, where $g$ is an automorphism or an anti-automorphism of $T$,
while in our definition and results about $\gamma,\delta$-endomorphisms we demanded that each of $\gamma$ and $\delta$ is conjugate by an automorphism or by an anti-automorphism of $A_1$, to $\alpha$.

However, one can see that if $f$ is an $\alpha,\delta$-endomorphism, 
where $\delta= h^{-1} \alpha h$, then it is enough to demand that $h$ is an automorphism or an anti-automorphism of the image of $f$, instead of demanding that $h$ is an automorphism or an anti-automorphism of $A_1$. 
\end{remark}
%%%%%%%%%%%%%%%%%

Now, instead of defining the involution $\gamma$ on $T$ ($\gamma(P)= Q$, $\gamma(Q)= P$), we could have defined on $T$, for example,
$\epsilon(P)= P$, $\epsilon(Q)= -Q$, extend in the obvious way to all of $T$, and get an involution on $T$.

Let $E$ be the involution on $A_1$ given by $E(X)= X$, $E(Y)= -Y$ ($E$ is just the involution $\beta$ we have already mentioned).
It is easy to see that $f E= \epsilon f$: 

$(f E)(X)= f(E(X))= f(X)= P= \epsilon(P)= \epsilon(f(X))= (\epsilon f)(X)$ and

$(f E)(Y)= f(E(Y))= f(-Y)= -Q= \epsilon(Q)= \epsilon(f(Y))= (\epsilon f)(Y)$.

Then taking $E$ and $\epsilon$ instead of $\alpha$ and $\gamma$, yield similar extension results, namely,
if the involution $\epsilon$ on $T$ can be extended to $A_1$,
where the extended involution is conjugate to $E$, 
then $f$ is an automorphism of $A_1$;
indeed, $f E= \epsilon f$. Let $\tilde{\epsilon}$ be the extension of $\epsilon$ to $A_1$, and let
$\tilde{\epsilon}= g^{-1} E g$, where $g$ is an automorphism or an anti-automorphism of $A_1$.
Clearly, $f E= \tilde{\epsilon} f$. 
We have already seen that $E$ is conjugate to $\alpha$ by $\varphi$: $E= \varphi^{-1} \alpha \varphi$
($\varphi(X):=\displaystyle\frac{X+Y}2$ and
$\varphi(Y):= Y-X$).
Therefore, $f E= \tilde{\epsilon} f$ becomes $f \varphi^{-1} \alpha \varphi= g^{-1} E g f$,
so $f \varphi^{-1} \alpha \varphi= g^{-1} (\varphi^{-1} \alpha \varphi) g f$.
Then, $(\varphi g f \varphi^{-1}) \alpha = \alpha (\varphi g f \varphi^{-1})$. 
\begin{itemize}
\item If $g$ is an automorphism of $A_1$, then $\varphi g f \varphi^{-1}$ is an $\alpha$-endomorphism of $A_1$, 
hence by \cite[Theorem 2.9]{moskowicz valqui} an automorphism of $A_1$. Then clearly $f$ is an automorphism of $A_1$.

\item If $g$ is an anti-automorphism of $A_1$, then $\varphi g f \varphi^{-1}$ is an $\alpha$-anti-endomorphism of $A_1$, hence by Proposition \ref{alpha anti-endomorphisms are onto} an anti-automorphism of $A_1$. Then clearly $f$ is an automorphism of $A_1$.

\end{itemize}

More generally, take any involution $E$ on $A_1$ that is conjugate to $\alpha$.
Write $E(X)= \sum \lambda_{ij}X^iY^j$ and $E(Y)= \sum \mu_{ij}X^iY^j$.
Then it is clear that $\epsilon(P):= \sum \lambda_{ij}P^iQ^j$ and $\epsilon(Q):= \sum \mu_{ij}P^iQ^j$ is an involution on $T$. 
It is easy to see that $f E= \epsilon f$: 

$(f E)(X)= f(\sum \lambda_{ij}X^iY^j)= \sum \lambda_{ij}f(X)^if(Y)^j=$ 

$\sum \lambda_{ij}P^iQ^j= \epsilon(P)= \epsilon(f(X))= (\epsilon f)(X)$ and

$(f E)(Y)= f(\sum \mu_{ij}X^iY^j)= \sum \mu_{ij}f(X)^if(Y)^j=$

$\sum \mu_{ij}P^iQ^j= \epsilon(Q)= \epsilon(f(Y))= (\epsilon f)(Y)$.

Again, if the involution $\epsilon$ on $T$ can be extended to $A_1$,
where the extended involution is conjugate to $E$, 
then $f$ is an automorphism of $A_1$ (just make the appropriate changes in the proof of the above example: 
$E(X)= X$, $E(Y)= -Y$, $\epsilon(P)= P$, $\epsilon(Q)= -Q$).

Therefore, we have: 

\begin{definition}[The generalized extension condition (first version)]
Let $f$ be an endomorphism of $A_1$.
We say that $f$ satisfies the generalized extension condition (first version), if there exists an involution $E$ on $A_1$ that is conjugate to $\alpha$
(write $E(X)= \sum \lambda_{ij}X^iY^j$ and $E(Y)= \sum \mu_{ij}X^iY^j$) such that its corresponding involution $\epsilon$ on $T$
($\epsilon(P):= \sum \lambda_{ij}P^iQ^j$ and $\epsilon(Q):= \sum \mu_{ij}P^iQ^j$)
can be extended to $A_1$, 
where the extended involution is conjugate to $E$. 
\end{definition}

\begin{theorem}\label{generalized theorem of extension condition first version}
Let $f$ be an endomorphism of $A_1$.
If $f$ satisfies the generalized extension condition (first version), then $f$ is an automorphism of $A_1$.
\end{theorem}

Similarly, instead of demanding that $\alpha$ restricted to $T$ is an involution on $T$ and $\gamma$ is conjugate by an automorphism or by an anti-automorphism of $T$ to the restriction of $\alpha$ to $T$, demand that $E$ restricted to $T$ is an involution on $T$ and $\epsilon$ is conjugate by an automorphism or by an anti-automorphism of $T$ to the restriction of $E$ to $T$.

Then taking $E$ and $\epsilon$ instead of $\alpha$ and $\gamma$, yield similar restriction results:

\begin{definition}[The generalized restriction condition (first version)]
Let $f$ be an endomorphism of $A_1$.
We say that $f$ satisfies the generalized restriction condition (first version), if there exists an involution $E$ on $A_1$ that is conjugate to $\alpha$,
such that $E$ restricted to $T$ is an involution on $T$. 
\end{definition}

\begin{theorem}\label{generalized theorem of restriction condition first version}
Let $f$ be an endomorphism of $A_1$.
If $f$ satisfies the generalized restriction condition (first version), then $f$ is an automorphism of $A_1$.
\end{theorem}

For a detailed proof, see the proof of Theorem \ref{generalized theorem of restriction condition second version} and adjust it to the first version.
%%%%%%%%%%%%%%%%%%%%%%%%%%%%%%%%%%%%%%%%%%%%%%%%%%%%%%%%%%%%%%%
\section{Equivalence of the Dixmier conjecture and the $\gamma,\delta$ conjecture for $A_1$ (second version)}

The reader has probably wondered why, in the previous section, ``first version" was added to most of our results.
The reason is that we wish to bring those results without the conjugacy conditions, where the conjugacy conditions are: ``each of $\gamma$ and $\delta$ is conjugate to $\alpha$", ``where the extended involution is conjugate to $\alpha$" and ``$\gamma$ is conjugate by an automorphism or by an anti-automorphism of $T$ to (the restriction of) $\alpha$".

The ``second version" results are just the first version results without the conjugacy conditions.

We now explain how to remove the conjugacy conditions, without loosing our previous, first version, results:  

An interesting problem is finding all the involutions on $A_1$; Of course, conjugates of $\alpha$ by an automorphism or by an anti-automorphism are involutions. Apriori, it may happen that there exist involutions other than conjugates of $\alpha$. However, in a private note \cite{bell}, J. Bell has sketched an idea of proof that there are no involutions on $A_1$ other than the conjugates of $\alpha$. His idea includes two steps; the first step has not yet been checked, though it seems to be true. The second step relies on the first step.

\begin{itemize}
\item First step: Recall that the group of automorphisms of $A_1$ is the amalgamated free product of two of its subgroups (denote these two subgroups by $A$ and $B$, where $A$ is the linear automorphisms and $B$ is the triangular automorphisms) over their intersection. %, see \cite{makar limanov free amalgamated}.
 
Similarly, it seems that the group $G$ generated by automorphisms and anti-automorphisms of $A_1$ is the amalgamated free product of two of its subgroups (denote these two subgroups by $\tilde{A}$ and $\tilde{B}$, each is defined similarly to the above $A$ and $B$) over their intersection.
(Notice that the group of automorphisms of $A_1$ is a noraml subgroup of $G$, having index $2$).

\item Second step: Applying Bass-Serre theorem \cite[Corollary 1, page 6]{trees} to a given involution $\gamma$ on $A_1$ shows almost immediately that it is conjugate to $\alpha$; indeed, $\gamma$ is of order $2$, so it is conjugate to an order $2$ element $a \in \tilde{A}$ or it is conjugate to an order $2$ element $b \in \tilde{B}$. Since there are no elements of order $2$ in $\tilde{B}$, $\gamma$ is conjugate to an order $2$ element $a \in \tilde{A}$. Finally, any order $2$ element $a \in \tilde{A}$ is conjugate to $\alpha$. 
\end{itemize}
%%%%%%%%%%%%%%%%%%%%
%%%%%%%%%%%%%%%%%%%%
%%%%%%%%%%%%%%%%%%%%
If indeed there are no involutions on $A_1$ other than the conjugates of $\alpha$, then in our first version results we can omit the conjugacy conditions and get the following results:

\begin{conjecture}[The $\gamma,\delta$ conjecture (second version)]
For every endomorphism $f$ of $A_1$, there exist involutions $\gamma$ and $\delta$, such that $f$ is a $\gamma,\delta$-endomorphism of $A_1$.
\end{conjecture}
%%%%%%%%%%%%%%%%%%%%%%%%
\begin{lemma}
Let $f$ be an endomorphism of $A_1$.
Then: $f$ is a $\gamma,\delta$-endomorphism of $A_1$ $\Longleftrightarrow$ $f$ is an automorphism of $A_1$.
\end{lemma}
%%%%%%
The following theorem shows that the Dixmier conjecture is true in the category of rings with involution:
%%%%%%%%%%
\begin{theorem}[The involutive Dixmier conjecture is true (second version)]
Let $f$ be an endomorphism from the ring with involution $(A_1,\gamma)$ to the ring with involution $(A_1,\delta)$. Then $f$ is an automorphism.
\end{theorem}
%%%%%%%%%%%%%%%%%%%
\begin{theorem}[Equivalence of the Dixmier conjecture and the $\gamma,\delta$ conjecture (second version)]
The Dixmier conjecture is true $\Longleftrightarrow$ The $\gamma,\delta$ conjecture (second version) is true.
\end{theorem}
                    
\begin{definition}[The extension condition (second version)]
Let $f$ be an endomorphism of $A_1$. 
We say that $f$ satisfies the extension condition (second version) if the involution $\gamma$ on $T$ ($\gamma(P)= Q$, $\gamma(Q)= P$) can be extended to $A_1$.
\end{definition}

\begin{theorem}\label{theorem of extension condition second version}
Let $f$ be an endomorphism of $A_1$.
If $f$ satisfies the extension condition (second version), then $f$ is an automorphism of $A_1$.
\end{theorem}

%%%%%%%%%%%%%%%%%%
\begin{definition}[The generalized extension condition (second version)]
Let $f$ be an endomorphism of $A_1$.
We say that $f$ satisfies the generalized extension condition (second version), if there exists an involution $E$ on $A_1$
(write $E(X)= \sum \lambda_{ij}X^iY^j$ and $E(Y)= \sum \mu_{ij}X^iY^j$) such that its corresponding involution $\epsilon$ on $T$
($\epsilon(P):= \sum \lambda_{ij}P^iQ^j$ and $\epsilon(Q):= \sum \mu_{ij}P^iQ^j$)
can be extended to $A_1$.
\end{definition}

%The generalized version of Theorem \ref{theorem of extension condition second version} is:

\begin{theorem}\label{generalized theorem of extension condition first version}
Let $f$ be an endomorphism of $A_1$.
If $f$ satisfies the generalized extension condition (second version), then $f$ is an automorphism of $A_1$.
\end{theorem}

%%%%%%%%%%%%%%%%%%%%%%%%%%%%%%%%%
We move to the restriction condition results; in contrast to the extension condition results,
here we have found a condition on $f$ which implies that $f$ satisfies the restriction condition (hence $f$ is an automorphism), namely
that one of $f(X)$,$f(Y)$ is symmetric or skew-symmetric with respect to any involution on $A_1$,
see Theorem \ref{important corollary}.
However, we bring a second proof for Theorem \ref{important corollary}, which is independent of the generalized restriction condition. 

\begin{definition}[The restriction condition (second version)]
Let $f$ be an endomorphism of $A_1$. 
We say that $f$ satisfies the restriction condition (second version) if the exchange involution $\alpha$ on $A_1$, when restricted to $T$, is an involution on $T$. 
\end{definition}

\begin{theorem}\label{theorem of restriction condition second version}
Let $f$ be an endomorphism of $A_1$.
If $f$ satisfies the restriction condition (second version), then $f$ is an automorphism of $A_1$.
\end{theorem}

\begin{proof}
$f$ satisfies the restriction condition (second version), so $\alpha$ restricted to $T$ is an involution on $T$.
Since $T$ is isomorphic to $A_1$ and we ``assume" that Bell's first step is indeed true, it follows that every involution on $T$ is conjugate to one chosen involution on $T$.

Let the restriction of $\alpha$ to $T$ be the ``chosen" involution on $T$.
Since $\gamma$ is an involution on $T$, we get that $\gamma$ is conjugate to the restriction of $\alpha$ to $T$:
$\gamma= g^{-1} \alpha g$, where $g$ is an automorphism or an anti-automorphism of $T$ ($\alpha$ is the restriction of $\alpha$ to $T$).

Therefore, we get that $f$ is an $\alpha,\gamma$-endomorphism of $A_1$ 
($\gamma$ is conjugate to $\alpha$), because:
$(f \alpha)(X)= f(Y)= Q= \gamma(P)= \gamma(f(X))= (\gamma f)(X)$ and
$(f \alpha)(Y)= f(X)= P= \gamma(Q)= \gamma(f(Y))= (\gamma f)(Y)$.

Hence, from Lemma \ref{gamma,delta is onto} we get that $f$ is an automorphism of $A_1$.
\end{proof}
%%%%%%%%%%%%%%

\begin{definition}[The generalized restriction condition (second version)]\label{definition generalized restriction second}
Let $f$ be an endomorphism of $A_1$.
We say that $f$ satisfies the generalized restriction condition (second version), if there exists an involution $E$ on $A_1$,
such that $E$ restricted to $T$ is an involution on $T$. 
\end{definition}

\begin{theorem}\label{generalized theorem of restriction condition second version}
Let $f$ be an endomorphism of $A_1$. 
If $f$ satisfies the generalized restriction condition (second version), then $f$ is an automorphism of $A_1$.
\end{theorem}

\begin{proof}
$f$ satisfies the generalized restriction condition (second version), so there exists an involution $E$ on $A_1$
such that $E$ restricted to $T$ is an involution on $T$. 
Since $T$ is isomorphic to $A_1$ and we ``assume" that Bell's first step is indeed true, it follows that every involution on $T$ is conjugate to one chosen involution on $T$.

Let the restriction of $E$ to $T$ be the ``chosen" involution on $T$.

Write $E(X)= \sum \lambda_{ij}X^iY^j$ and $E(Y)= \sum \mu_{ij}X^iY^j$. 
Its corresponding involution $\epsilon$ on $T$ is defined to be $\epsilon(P):= \sum \lambda_{ij}P^iQ^j$ and 
$\epsilon(Q):= \sum \mu_{ij}P^iQ^j$.

Since $\epsilon$ is an involution on $T$, we get that $\epsilon$ is conjugate to the restriction of $E$ to $T$:
$\epsilon= g^{-1} E g$, where $g$ is an automorphism or an anti-automorphism of $T$ ($E$ is the restriction of $E$ to $T$).

Therefore, we get that $f$ is an $E,\epsilon$-endomorphism of $A_1$ because:

$(f E)(X)= f(\sum \lambda_{ij}X^iY^j)= \sum \lambda_{ij}f(X)^if(Y)^j=$ 

$\sum \lambda_{ij}P^iQ^j= \epsilon(P)= \epsilon(f(X))= (\epsilon f)(X)$ and

$(f E)(Y)= f(\sum \mu_{ij}X^iY^j)= \sum \mu_{ij}f(X)^if(Y)^j=$

$\sum \mu_{ij}P^iQ^j= \epsilon(Q)= \epsilon(f(Y))= (\epsilon f)(Y)$.

Now $E$ is conjugate to $\alpha$ by our ``assumption" in this section that there are no involutions on $A_1$ other than the conjugates of $\alpha$, and $\epsilon= g^{-1} E g$. Write $E= h^{-1} \alpha h$, where $h$ is an automorphism or an anti-automorphism of $A_1$.
Then $f E= \epsilon f$ becomes $f h^{-1} \alpha h= g^{-1} E g f= g^{-1} (h^{-1} \alpha h) g f= g^{-1} h^{-1} \alpha h g f$,
so $f h^{-1} \alpha h= g^{-1} h^{-1} \alpha h g f$.
Then $h g f h^{-1} \alpha= \alpha h g f h^{-1} $: \begin{itemize}

\item If $g$ is an automorphism of $T$, then $h g f h^{-1}$ is an $\alpha$-endomorphism of $A_1$, 

so \cite[Theorem 2.9]{moskowicz valqui} implies that $h g f h^{-1}$ is an ($\alpha$-)automorphism of $A_1$.
Therefore, $f$ is an automorphism of $A_1$.

\item If $g$ is an anti-automorphism of $T$, then $h g f h^{-1}$ is an $\alpha$-anti-endomorphism of $A_1$, 
so Proposition \ref{alpha anti-endomorphisms are onto} implies that 
$h g f h^{-1}$ is an ($\alpha$-)anti-automorphism of $A_1$.

Therefore, $f$ is an automorphism of $A_1$.
\end{itemize}

%Observe that the fact that $g$ is an automorphism or an anti-automorphism of $T$ (and not of $A_1$; apriori, $T$ may be strictly contained in $A_1$) is good enough, see Remark \ref{remark used in generalized}.
\end{proof}

%%%%%%%%%%%%%%
Let $\epsilon$ be an involution on $A_1$ and let $w \in A_1$. $w$ is symmetric if $\epsilon(w)= w$, and $w$ is skew-symmetric if $\epsilon(w)= -w$.

Now, in case one of $f(X),f(Y)$ is symmetric or skew-symmetric with respect to $\alpha$, the restriction condition is satisfied by $f$:

\begin{theorem}\label{before important corollary}
Let $f$ be an endomorphism of $A_1$. 
Assume that one of the following conditions is satisfied:\begin{itemize}
\item $P$ is symmetric.
\item $P$ is skew-symmetric.
\item $Q$ is symmetric.
\item $Q$ is skew-symmetric.
\end{itemize}
Where by symmetric or skew-symmetric we mean symmetric or skew-symmetric with respect to $\alpha$.
Then $f$ is an automorphism of $A_1$.
\end{theorem}

The commutator of two symmetric or two skew-symmetric elements is skew-symmetric, hence it is impossible to have both $P$ and $Q$ symmetric or both $P$ and $Q$ skew-symmetric (since $[Q,P]= 1$ and $1$ is not skew-symmetric).
But it may happen that $P$ is symmetric and $Q$ is skew-symmetric or vice-versa, for example:
$P= \displaystyle\frac{X+Y}2$ and
$Q= Y-X$.

Actually, if $P$ is symmetric and $Q$ is skew-symmetric (or vice-versa), then it is immediate that $f$ is an automorphism; 
Write $P= S$ and $Q= K$ with $S$ symmetric and $K$ skew-symmetric. 
Define $g(X)= (1/\sqrt{2})(S-K)$ and 
$g(Y)=  (1/\sqrt{2})(S+K)$.

$g$ is an endomorphism of $A_1$:
 
$[g(Y),g(X)]= [(1/\sqrt{2})(S+K),(1/\sqrt{2})(S-K)]= (1/2)[S+K,S-K]=$

$(1/2)[S,-K]+(1/2)[K,S]= (1/2)[K,S]+(1/2)[K,S]= [K,S]= [Q,P]= 1$.

$g$ is an $\alpha$-endomorphism of $A_1$: 

$(g \alpha)(X)= g(Y)=  (1/\sqrt{2})(S+K)= (1/\sqrt(2))\alpha(S-K)= \alpha((1/\sqrt{2})(S-K))= \alpha(g(X))= (\alpha g)(X)$,

$(g \alpha)(Y)= g(X)=  (1/\sqrt{2})(S-K)= (1/\sqrt(2))\alpha(S+K)= \alpha((1/\sqrt{2})(S+K))= \alpha(g(Y))= (\alpha g)(Y)$.

{}From \cite[Theorem 2.9]{moskowicz valqui} $g$ is an automorphism of $A_1$, hence it is clear that $f$ is an automorphism of $A_1$.
%%%%%%%%%%%%%%%
\begin{proof}
Assume that $P$ is symmetric, namely $\alpha(P)= P$.
{}From $[Q,P]= 1$ we get $[P,\alpha(Q)]= [\alpha(P),\alpha(Q)]= \alpha([Q,P])= \alpha(1)= 1$.
Then $[P,-Q-\alpha(Q)]= [P,-Q]-[P,\alpha(Q)]= 1-1= 0$, so $-Q-\alpha(Q)$ is in the centralizer of $P$.
The centralizer of $P$ is the polynomial algebra $K[P]$, see \cite{cent}, 
hence there exists a polynomial $h(t) \in K[t]$ such that $-Q-\alpha(Q)= h(P)$, 
so $\alpha(Q)= -Q-h(P)$.
This implies that $\alpha$ restricted to $T$ is an involution on $T$; indeed,
$\alpha(P)= P \in T$ and $\alpha(Q)= -Q-h(P) \in T$, so $\alpha$ restricted to $T$ is an endomorphism of $T$.
Then it is clear that $\alpha$ restricted to $T$ is an involution on $T$; just for fun:

$\alpha$ restricted to $T$ is of order $2$:
$\alpha(\alpha(Q))= \alpha(-Q-h(P))= \alpha(-Q)+\alpha(-h(P))= -\alpha(Q)-\alpha(h(P))=$
$-(-Q-h(P))-h(P)= Q+h(P)-h(P)= Q$ (of course $\alpha(h(P))= h(P)$, since $\alpha(P)= P$).

Therefore, $f$ satisfies the restriction condition (second version), so from Theorem \ref{theorem of restriction condition second version} $f$ is an automorphism of $A_1$.

Showing that each of the other three conditions implies that $f$ is an automorphism of $A_1$, is similar.
\end{proof}

\begin{theorem}\label{important corollary}
Let $f$ be an endomorphism of $A_1$. 
Assume that one of the following conditions is satisfied:\begin{itemize}
\item $P$ is symmetric.
\item $P$ is skew-symmetric.
\item $Q$ is symmetric.
\item $Q$ is skew-symmetric.
\end{itemize}
Where by symmetric or skew-symmetric we mean symmetric or skew-symmetric with respect to any involution on $A_1$.
Then $f$ is an automorphism of $A_1$.
\end{theorem}

\begin{proof}
Assume that $P$ is symmetric with respect to $E= g^{-1}\alpha g$ ($g$ is an automorphism or an anti-automorphism of $A_1$), namely $E(P)= P$.
{}From $[Q,P]= 1$ we get $[P,E(Q)]= [E(P),E(Q)]= E([Q,P])= E(1)= 1$.
Then $[P,-Q-E(Q)]= [P,-Q]-[P,E(Q)]= 1-1= 0$, so $-Q-E(Q)$ is in the centralizer of $P$.
The centralizer of $P$ is the polynomial algebra $K[P]$, see \cite{cent}, 
hence there exists a polynomial $h(t) \in K[t]$ such that $-Q-E(Q)= h(P)$, 
so $E(Q)= -Q-h(P)$.
This implies that $E$ restricted to $T$ is an involution on $T$; indeed,
$E(P)= P \in T$ and $E(Q)= -Q-h(P) \in T$, so $E$ restricted to $T$ is an endomorphism of $T$.
Then it is clear that $E$ restricted to $T$ is an involution on $T$; just for fun:

$E$ restricted to $T$ is of order $2$:
$E(E(Q))= E(-Q-h(P))= E(-Q)+E(-h(P))= -E(Q)-E(h(P))=$
$-(-Q-h(P))-h(P)= Q+h(P)-h(P)= Q$ (of course $E(h(P))= h(P)$, since $E(P)= P$).

Therefore, $f$ satisfies the generalized restriction condition (second version), so from Theorem \ref{generalized theorem of restriction condition second version} $f$ is an automorphism of $A_1$.

Showing that each of the other three conditions implies that $f$ is an automorphism of $A_1$, is similar.

\textbf{Second proof, independent of the generalized restriction condition}:

Assume that $P$ is symmetric with respect to $E= g^{-1}\alpha g$ ($g$ is an automorphism or an anti-automorphism of $A_1$), namely $E(P)= P$.

Write: $P= S_1$ and $Q= S_2+K_2$, where $S_1$ and $S_2$ are symmetric and $K_2$ is skew-symmetric, with respect to $E$.

{}From $1= [Q,P]$ we get $1= [S_2+K_2,S_1]= [S_2,S_1]+[K_2,S_1]$.
$[S_2,S_1]$ is skew-symmetric and $[K_2,S_1]$ is symmetric, hence
$[S_2,S_1]= 0$ and $[K_2,S_1]= 1$.

$[S_2,S_1]= 0$, so $S_2$ is in the centralizer of $S_1$. 
The centralizer of $S_1$ is the polynomial algebra $K[S_1]$, see \cite{cent}, 
hence there exists a polynomial $h(t) \in K[t]$ such that $S_2= h(S_1)$.

$[K_2,S_1]= 1$, so $g(X):= (1/\sqrt(2))(S_1-K_2)$ and
$g(Y):= (1/\sqrt(2))(S_1+K_2)$ defines an $E$-endomorphism of $A_1$, which is an automorphism of $A_1$
(by Lemma \ref{gamma,delta is onto}). Then obviously $S_1$ and $K_2$ generate all $A_1$.

Of course, $K_2= S_2+K_2-S_2= Q-h(S_1)= Q-h(P)$ is in $T$.
So $S_1$ and $K_2$ are in $T$, which implies that $T= A_1$, hence $f$ is an automorphism of $A_1$.

Showing that each of the other three conditions implies that $f$ is an automorphism of $A_1$, is similar.

\end{proof}
%%%%%%%%%%%%%%
The first condition in Theorem \ref{important corollary} is equivalent to the condition that there exists an automorphism or an anti-automorphism $g$ of $A_1$ such that $g(P)$ is symmetric with respect to $\alpha$ (and similarly the other three conditions).

\begin{proposition}\label{prop E alpha}
Let $f$ be an endomorphism of $A_1$.
The following conditions are equivalent:\begin{itemize}
\item [(1)] There exists an involution $E$ on $A_1$ such that $P$ is symmetric (or skew-symmetric) with respect to $E$.
\item [(2)] There exists an automorphism or an anti-automorphism $g$ of $A_1$ such that $g(P)$ is symmetric (or skew-symmetric) with respect to $\alpha$.
\end{itemize}
\end{proposition}

Of course, the analogue result when taking $Q$ instead of $P$, is also true.

\begin{proof}
$(1)\Rightarrow(2)$: Let $E$ be an involution on $A_1$ such that $P$ is symmetric with respect to $E$, namely
$E(P)= P$. By our ``assumption" $E$ is conjugate to $\alpha$, so there exists an automorphism or an anti-automorphism $g$ of $A_1$ such that $E= g^{-1} \alpha g$. Hence, $E(P)= P$ becomes $(g^{-1} \alpha g)(P)= P$, so
$\alpha (g(P))= g(P)$, which means that $g(P)$ is symmetric with respect to $\alpha$.

The skew-symmetric case is similar.

$(2)\Rightarrow(1)$: Let $g$ be an automorphism or an anti-automorphism of $A_1$ such that $g(P)$ is symmetric with respect to $\alpha$, namely $\alpha (g(P))= g(P)$. 
Then $(g^{-1} \alpha g)(P)= P$, which means that $P$ is symmetric with respect to the involution $g^{-1} \alpha g$ on $A_1$.

The skew-symmetric case is similar.
\end{proof}

{}From Proposition \ref{prop E alpha}, given an endomorphism $f$ of $A_1$, one wishes to find linear and triangular automorphisms and linear and triangular anti-automorphisms $g_1,\ldots,g_n$ such that $(g_1 \cdots g_n)(P)$ is symmetric or skew-symmetric with respect to $\alpha$.
(Or one wishes to find linear and triangular automorphisms and linear and triangular anti-automorphisms $g_1,\ldots,g_n$ such that $(g_1 \cdots g_n)(Q)$ is symmetric or skew-symmetric with respect to $\alpha$).

%%%%%%%%%%%%%%
Next, the second version of the $g,h$ conjecture is exactly the same as the first version of the $g,h$ conjecture:   

\begin{conjecture}[The $g,h$ conjecture (second version)]\label{second version of the g,h conjecture}
For every endomorphism $f$ of $A_1$, there exist $g$ and $h$, both automorphisms, both anti-automorphisms or one is an automorphism and the other is an anti-automorphism, such that $h f g^{-1}$ is an $\alpha$-endomorphism (when both $g$ and $h$ are automorphisms or both are anti-automorphisms) or an $\alpha$-anti-endomorphism (when one of $g$ and $h$ is an automorphism and the other is an anti-automorphism).
\end{conjecture}

\begin{theorem}[Equivalence of the Dixmier conjecture and the $g,h$ conjecture (second version)]\label{dixmier iff g,h second version}
The Dixmier conjecture is true $\Longleftrightarrow$ The $g,h$ conjecture (second version) is true.
\end{theorem}          
%%%%%%%%%%%%%%%%%%%%%%%%%%%%%%%
%%%%%%%%%%%%%%%%%%%%%%%%%%%%%%%
\section{Analogous results for $K[x,y]$}

The above Bell's idea concerning involutions on $A_1$ can be applied to $K[x,y]$.

Apriori, we define an involution $\gamma$ on $k[x,y]$ as an automorphism of order $2$, without demanding that the Jacobian of $\gamma(x)$, $\gamma(y)$ will equal $-1$ (of course, the Jacobian of $\gamma(x)$, $\gamma(y)$ is a non-zero scalar, since $\gamma$ is invertible). However, Bell's idea shows that $\gamma$ is conjugate to $\alpha$, hence we get for free that the Jacobian of $\gamma(x)$, $\gamma(y)$ equals $-1$ (since the Jacobian of $\alpha(x)$, $\alpha(y)$ equals $-1$).

\begin{proposition}\label{bell jacobian}
There are no involutions on $K[x,y]$ other than the conjugates of $\alpha$.
\end{proposition}

\begin{proof}
Since $K[x,y]$ is commutative, anti-automorphisms are automorphisms, so every involution on $K[x,y]$ is an order $2$ automorphism.

It is well-known that the group of automorphisms of $K[x,y]$ is the amalgamated free product of $A$ and $B$ over their intersection, where $A$ is the affine automorphisms and $B$ is the triangular automorphisms;
this result appeared in many articles whose authors are: van der Kulk, Kambayashi, Nagata, Dicks, Alperin and Shafarevich (the list of articles can be found in \cite[page 114]{van den essen}).

Therefore, one can immediately apply Bass-Serre theorem \cite[Corollary 1, page 6]{trees} to a given involution $\gamma$ on $K[x,y]$ and get that $\gamma$ is conjugate to $\alpha$; indeed, $\gamma$ is of order $2$, so it is conjugate to an order $2$ element $a \in A$ or it is conjugate to an order $2$ element $b \in B$. Since there are no elements of order $2$ in $B$, $\gamma$ is conjugate to an order $2$ element $a \in A$. Finally, any order $2$ element $a \in A$ is conjugate to $\alpha$. 
\end{proof}

We can obtain for $K[x,y]$ similar results to the results of $A_1$. We leave to the reader the task of checking this.

We just mention that it seems that now we have two advantages in comparison to $A_1$: 
\begin{itemize}
\item By Proposition \ref{bell jacobian}, we do know that there are no involutions on $A_1$ other than the conjugates of $\alpha$, so we can obtain the results for $K[x,y]$ immediately from the second version results of $A_1$.

In particular, every morphism $f$ from the ring with involution $(K[x,y],\gamma)$ to the ring with involution $(K[x,y],\delta)$ that satisfies $\Jac(f(x),f(y)) \in K^*$ is invertible.

\item The problem of extending an involution from $K[p,q]$ to $K[x,y]$ (given a morphism $f$ of $K[x,y]$ with $p= f(x)$ and $q= f(y)$) seems a little less difficult problem, since now we deal with commutative rings. However, the condition on the Jacobian may add some new difficulty, since we wish that the Jacobian of the extension will equal $-1$.
\end{itemize}

The Jacobian conjecture for $K[x,y]$ implies the Dixmier conjecture for $A_1$.
This was proved by Tsuchimoto \cite{tsuchimoto}, Belov and Kontsevich \cite{belov} and Bavula \cite{bavula jacobian}
(actually, the generalized version was proved in those papers).

Therefore, if one can prove that the $\gamma,\delta$ conjecture for $K[x,y]$ is true (hence by our results the Jacobian conjecture for $K[x,y]$ is true), then the dixmier conjecture for $A_1$ is true.

Finally, we suggest to check if all the results brought in this article are applicable to Poisson algebras.
If so, in view of \cite{adja}, it will be interesting to see if the $\gamma,\delta$ conjecture is true there. 

\section{Acknowledgements}
I wish to thank Prof. C. Valqui for working with me on the starred Dixmier conjecture \cite{moskowicz valqui}, Prof. J. Bell for his note \cite{bell} concerning involutions on $A_1$, and Prof. L. Rowen and Prof. U. Vishne for listening to my ideas and commenting on them.

%%%%%%%%%%%%%%%%%%%%%%%%%%%%%%%%%%%%%%%
\bibliographystyle{plain}

\end{document}